\documentclass[b5paper] {article}
[12pt]

\makeatletter
\renewcommand*\l@section{\@dottedtocline{1}{1.5em}{2.3em}}
\makeatother

\usepackage{amsfonts}
\usepackage{amssymb}
\usepackage[T1]{fontenc}

\usepackage{tikz}
\usetikzlibrary{calc}

\usepackage{CJK}
\usepackage{amsmath}

\usepackage{amsfonts}
\usepackage{amssymb}
\usepackage{amsthm}
\usepackage{amssymb}
\usepackage{enumerate}
\usepackage[calc]{picture}
\usepackage[all,cmtip]{xy}

\usepackage[mathscr]{eucal}
\usepackage{eqlist}

\usepackage{color}
\usepackage{abstract}
\usepackage[T1]{fontenc}

\usepackage[top=1.2cm, bottom=1.8cm, left=1.8cm, right= 1.8cm]{geometry}

\theoremstyle{plain}
\newtheorem{theorem}{Theorem}
\newtheorem{proposition}[theorem]{Proposition}
\newtheorem{lemma}[theorem]{Lemma}

\newtheorem{example}[theorem]{Example}
\newtheorem{corollary}[theorem]{Corollary}

\theoremstyle{definition}

\usepackage{etoolbox}
\newtheoremstyle{myrem}
 {3pt}
 {3pt}
 {\normalsize}
 { }
 {\itshape}
 {:}
 { }
 {}

 \theoremstyle{myrem}
 \newtheorem{remark}{Remark}
 \appto\remark{\leftskip\parindent}
 \appto\remark{\rightskip\parindent}

\numberwithin{equation}{section}
\numberwithin{theorem}{section}

\begin{document}

~~~

\vspace{1cm}

\begin{center}
{\Large {\textbf {Weighted  Analytic Torsion  for Weighted Digraphs}}} \footnote[1]{ Shiquan Ren was funded by the Natural Science Foundation of China (No.  12001310)  and China Postdoctoral Science Foundation (No. 2020M680494).
Chong Wang was funded by the Science and Technology Project of Hebei Education Department (No. QN2019333), the Natural Fund of Cangzhou Science and Technology Bureau (No.  197000002) and a Project of Cangzhou Normal University (No.  xnjjl1902).  }
 \vspace{0.58cm}\\

Shiquan Ren,  Chong Wang


\bigskip

\begin{quote}
\begin{abstract}

In 2020,  Alexander Grigor'yan,  Yong  Lin   and   Shing-Tung Yau  \cite{lin1} introduced  the Reidemeister torsion and the analytic torsion 
 for  digraphs  by  means of the path complex and the  path homology theory.  Based on 
 the analytic torsion 
 for  digraphs  introduced in \cite{lin1}, 
we     consider  the notion of   weighted analytic torsion for  vertex-weighted digraphs.
For  any   non-vanishing real functions $f$ and $g$ on the vertex set, we consider the    vertex-weighted digraphs with  the  weights $(f,g)$.    We  calculate the $(f,g)$-weighted  analytic  torsion  by examples  and prove that the $(f,g)$-weighted analytic  torsion  only depend on the ratio $f/g$.  In particular, if the weight is of the diagonal form $(f,f)$,  then the weighted  analytic  torsion equals to the usual (un-weighted)  torsion.
\end{abstract}
\end{quote}

\end{center}

\smallskip

\begin{quote}

\end{quote}

\bigskip

\section{Introduction}

\subsection{Torsions of Riemannian Manifolds}

The theory of R-torsion and  analytic torsion of Riemannian manifolds was developed during the 20-th century by K. Reidemeister \cite{rei}, John Milnor \cite{milnor},  D. B. Ray and I. Singer \cite{ray}, Werner M\"{u}ller \cite{muller},  J. Cheeger \cite{ch},   etc.   Let  $M$  be a  compact  oriented  Riemannian $m$-manifold  with a smooth triangulation $K$.
Let  $\rho: \pi_1(K)\longrightarrow {\rm O}(E)$    be  an  acyclic  orthogonal  representation  on a  finite-dimensional  real or  complex  vector space $E$.
We  have a chain complex
\begin{eqnarray*}
C(K,\rho)=\sum_{q=0}^m C_q(\tilde K)\otimes_{\mathbb{F}(\pi_1(K))} E
\end{eqnarray*}
where $\tilde K$  is the  universal cover of $K$  and  $\mathbb{F}$  is the field $\mathbb{R}$ of real numbers if $E$  is a real vector space  or the field  $\mathbb{C}$   of complex numbers  if  $E$  is a complex vector space.  For each  non-negative  integer $q$,  let $\partial^\rho_q:  C_q(K,\rho)\longrightarrow C_{q-1}(K,\rho)$ be the boundary map  twisted by $\rho$.  Under a preferred basis for each $C_q(K,\rho)$,  $\partial^\rho_q$  is represented  by a real matrix.  Let $(\partial^\rho_q)^*$  be the transpose matrix  and  let the $\rho$-twisted Hodge-Laplace operator be
\begin{eqnarray*}
\Delta^{(c)}_q=(\partial^\rho_q)^*\partial_q^\rho+ \partial^\rho_{q+1}(\partial^\rho_{q+1})^*.
 \end{eqnarray*}
 Since  $\rho$  is assumed acyclic,  $\Delta^{(c)}_q$  is  non-degenerate for each $q$.  The  Reidemeister-Franz torsion (or R-torsion)  $\tau_\rho(M)$
 is a combinatorial invariant of $M$ and $\rho$ given by (cf.  \cite[Proposition~1.7]{ray})
 \begin{eqnarray*}
 \log  \tau_\rho(M)=\frac{1}{2}\sum _{q=0}^m (-1)^{q+1} q \log \det (\Delta_q^{(c)}).
 \end{eqnarray*}
 On the other hand,  let $E(\rho)$  be the associated vector bundle $E\longrightarrow  \tilde M\otimes_{\mathbb{F}(\pi_1(M))} E\longrightarrow M$ over $M$  where $\tilde M$  is the universal cover of $M$.   Let $\mathcal{D}=\sum_{q=0}^{m}\mathcal{D}^q$  be the vector space of smooth differential forms on $M$  with values in $E(\rho)$.   For each  $q=0,1,\ldots,m$,  let $d: \mathcal{D}^q\longrightarrow \mathcal{D}^{q+1}$  be the exterior differential,  $*: \mathcal{D}^q\longrightarrow \mathcal{D}^{m-q}$  be the duality with respect to the Riemannian metric on $M$, and $\delta=(-1)^{Nq+N+1}*d*$ be the formal adjoint of $d$.   The Hodge-Laplace operator  is
 \begin{eqnarray*}
 \Delta_q=\delta  d+d\delta.
  \end{eqnarray*}
  Since $\rho$  is assumed acyclic,  by the Hodge theorem, zero is not an eigenvalue of $\Delta_q$.   The zeta function of $\Delta_q$  is
\begin{eqnarray*}
\zeta_{q,\rho}(s)= \frac{1}{\Gamma(s)}\int_0^\infty t^{s-1} {\rm Tr}(e^{-t\Delta_q}) dt
\end{eqnarray*}
and the analytic torsion $T_\rho(M)$,  which depends on $M$ and $\rho$,   is defined as the positive real root of
\begin{eqnarray*}
\log  T_{\rho}(M)=\frac{1}{2}\sum _{q=0}^m (-1)^q q \zeta'_{q,\rho}(0).
\end{eqnarray*}
  It  is  proved by Cheeger \cite{ch}  and  M\"uller \cite{muller}  independently  that  $\tau_\rho(M)=T_\rho(M)$.

\smallskip

\subsection{Torsions of Digraphs}


In recent decades,  the path complex  and the  path-homology theory for digraphs have been    developed  (cf. \cite{lin1,lin2,lin3,lin4,lin5,lin6}).

Let $V$ be a finite set.    Let $n$ be a nonnegative integer.  An  {\it elementary $n$-path} is an ordered sequence $v_0v_1\ldots v_n$ of $n+1$ vertices in $V$.  Here the vertices $v_j$ and $v_k$ are not required to be distinct for any  integers  $0\leq j<k\leq n$.   A  linear combination of   elementary $n$-paths on $V$ is called an  {\it $n$-path} on $V$.  Let $\Lambda _n(V)$  be the vector space consisting of all the $n$-paths on $V$.  The {\it $n$-th boundary operator} is a linear map $\partial_n: \Lambda_n(V)\longrightarrow\Lambda_{n-1}(V)$  given by  $\partial_n(v_0\ldots v_n)=\sum_{i=0}^n (-1)^i v_0\ldots \widehat{v_i} \ldots v_n$.   It can be verified $\partial_{n}\partial_{n+1}=0$  for any $n$ be a nonnegative integer.  Thus equipped with these boundary operators,   we have a chain complex $\Lambda_*(V)$  (cf. \cite[Subsection~2.1]{lin1}, \cite{lin2}, \cite{lin6}, etc.).

An elementary $n$-path $v_0\ldots v_n$ is called {\it regular} if $v_{i-1}\neq v_i$  for all integers $1\leq i\leq n$  and is called {\it irregular} otherwise.  Let $I_n(V)$ be the subspace of $\Lambda_n(V)$ spanned by all the  irregular elementary $n$-paths.  Let $\mathcal{R}_n(V)=
\Lambda_n(V)/I_n(V)$.  The  {\it $n$-th  irregular boundary operator}  is  a linear map
 $\partial_n: \mathcal{R}_n(V)\longrightarrow \mathcal{R}_{n-1}(V)$  which  is obtained  from the $n$-th boundary  operator of $\Lambda_*(V)$ by  moduloing the terms in $I_{n-1}(V)$.
 It is proved that $\mathcal{R}_*(V)$  is a   chain complex with the regular boundary operators $\partial_*$  (cf. \cite[Subsection~2.1]{lin1}, \cite{lin2}, \cite{lin6}, etc.).

  A  {\it digraph} $G=(V,E)$ is a couple where $V$ is a finite set and $E$ is a subset of  $\{(u,v)\in V\times V\mid u\neq  v\}$.  An {\it allowed elementary $n$-path} on $G$ is an elementary $n$-path $v_0v_1\ldots v_n$  on $V$  such  that  $(v_{i-1},v_i)\in E$ for each integer $1\leq i\leq n$.   A linear combination of allowed elementary $n$-paths on $G$ is called an {\it allowed $n$-path} on $G$.  The vector space consisting of all the  allowed elementary $n$-paths on $G$ is denoted as $\mathcal{A}_n(G)$.   The  graded vector  space  $\mathcal{A}_*(G)$  is a graded linear  sub-space of $\mathcal{R}_*(V)$.  However, $\mathcal{A}_*(G)$ may not be a sub-chain complex of $\mathcal{R}_*(V)$.  An $n$-path $\omega\in \mathcal{A}_n(G)$ is called {\it $\partial$-invariant}  if $\partial_n \omega\in \mathcal{A}_{n-1}(G)$.  The space of all the  $\partial$-invariant $n$-paths on $G$ is  given by
\begin{eqnarray}\label{eq-intro}
\Omega_n(G)=\mathcal{A}_n(G)\cap\partial_n^{-1}(\mathcal{A}_{n-1}(G)).
\end{eqnarray}
Let $n$ run over $0,1,2,\ldots$   Then     (\ref{eq-intro})   gives a sub-chain complex $\Omega_*(G)$  of $\mathcal{R}_*(V)$.  The {\it $n$-th path homology} $H_n(G)$  of $G$ is defined as the $n$-th homology group of the chain complex $\Omega_*(G)$.

\smallskip




  In 2020,  Alexander Grigor'yan,  Yong  Lin  and  Shing-Tung Yau \cite{lin1} defined the R-torsion and  the   analytic torsion for digraphs by using path complexes and proved that the R-torsion and  the  analytic torsion are equal (cf. \cite[Theorem~3.14]{lin1}).

Let  $G$  be a digraph such  that there exists a  nonnegative  integer  $N$  satisfying  $\Omega_n(G)=0$  for any integer  $n\geq N+1$.
 For any  $u=\sum_i u^i e_i$ and $v=\sum_i v^i e_i$  in $\Omega_n(G)$   where  $e_i$ are allowed elementary paths on $G$, the first (standard)  inner  product  of $u$  and $v$  is $\langle u,v\rangle =\sum_i u^i v^i$  and  the second  (normalized)  inner  product  of $u$  and $v$  is $\langle u,v\rangle =\frac{1}{n!}\sum_i u^i v^i$.   In the following,  we  choose either the first  inner products or the second inner products  uniformly, written $\langle~,~\rangle$.
  Let $\partial^*$   be  the  adjoint of  $\partial$  with respect to the chosen inner product.  The Hodge-Laplace operator  $\Delta_n: \Omega_n(G)\longrightarrow  \Omega_n(G)$    is  given by
  \begin{eqnarray*}
  \Delta_n=\partial_n^*\partial_n+\partial_{n+1}\partial_{n+1}^*.
   \end{eqnarray*}
   Let  $\{\lambda_i\}_{i=1}^{\dim\Omega_n(G)}$   be  the  eigenvalues  of $\Delta_n$.    The zeta  function $\zeta_n$  of   $\Delta_n$  is   defined by
 \begin{eqnarray*}
 \zeta_n(s)=\sum_{\lambda_i>0}\frac{1}{\lambda_i^s}.
 \end{eqnarray*}
 The  analytic torsion  $T(G)$  of  $G$  is  given  by
 \begin{eqnarray*}
 \log  T(G)=\frac{1}{2}\sum_{n=0}^N(-1)^n n\zeta'_n (0).
 \end{eqnarray*}
On  the  other hand,  the  R-torsion  $\tau(G)$ of  $G$  is defined to be
 \begin{eqnarray*}
 \tau(G)=\tau(\Omega_*(G),\langle~,~\rangle)
 \end{eqnarray*}
 which  is  given by (\ref{eq-lapq})  with respect to the chain complex $\Omega_*(G)$  and the inner products  $\langle~,~\rangle$.  By    an extension of the argument of \cite[Proposition~1.7]{ray},  it  is  proved  in  \cite[Theorem~3.14]{lin1}  that  $\tau(G)=T(G)$.

\smallskip

\subsection{Our Results}

In recent years,  motivated by the fast  development of topological data analysis,  the weighted persistent homology of weighted  simplicial complexes  has  been studied in \cite{cy1,cy,cy2} and the   persistent path-homology  of vertex-weighted digraphs has been  studied in \cite{wch}.   The weighted boundary operator  plays a key role  in  the weighted persistent homology \cite{cy1,cy,cy2} of weighted simplicial complexes and the persistent path-homology  of  vertex-weighted digraphs  \cite{wch}.

\smallskip

In  this paper,  we consider some weights on the vertex set of a digraph. We  construct the weighted Hodge-Laplace operators  by using the weighted boundary operators and their adjoints.  We  study the weighted analytic torsion which is  given  in   (\ref{eq-baqkl})   and  is determined   by the zeta functions of the positive eigenvalues of the weighted Hodge-Laplace operators.

\smallskip

Let $G=(V,E)$  be a digraph.   Let   $f, g:  V\longrightarrow \mathbb{R}^\times$  be  non-vanishing real functions  on   $V$.    Let $n$ be a nonnegative integer.
The $f$-weighted boundary map
$
\partial_n^f: \Lambda _n(V)\longrightarrow \Lambda _{n-1}(V)
$
 is a  linear map  sending each elementary $n$-path
$v_0v_1\ldots v_n$  to  a formal alternating sum $\sum_{i=0}^n (-1)^i f(v_i) v_0\ldots \widehat{v_i}
\ldots v_n $  of  elementary $(n-1)$-paths.  It  induces a $f$-weighted  boundary map $
\partial_n^f: \mathcal{R}_n(V)\longrightarrow \mathcal{R}_{n-1}(V)
$ by  moduloing the terms  in $I_{n-1}(V)$  of the alternating sum.
 Recall that  the space  $\mathcal{A}_n(G)$  consists of all the allowed $n$-paths  on $G$.   It is a subspace of $\mathcal{R} _n(V)$  and is spanned by  all the allowed elementary $n$-paths
\begin{eqnarray*}
v_0 v_1\ldots v_n, ~~~ (v_{i-1}, v_i)\in E{\rm ~for~each~integer~} 1\leq i\leq n.
\end{eqnarray*}
The $(f,g)$-weighted $\partial$-invariant   complex of $G$ is a chain complex
\begin{eqnarray*}
\Omega^{f,g}_n(G)=\mathcal{A}_n(G)\cap(\partial_n^f)^{-1}\big(\mathcal{A}_{n-1}(G)\big),  ~~~n=0,1,2,\ldots,
\end{eqnarray*}
with the weighted boundary operators $\{\partial_n^f\mid  n=0,1,2,\ldots\}$  and with the inner products     $\langle~,~\rangle_g$  where for each nonnegative integer $n$,
 the     inner product     $\langle~,~\rangle_g$  on  $\Omega_n^{f,g}(G)$  is    a  real-valued   bilinear  map  on  $\Omega_n^{f,g}(G)\times \Omega_n^{f,g}(G)$    given by
\begin{eqnarray*}
&&\Big\langle \sum_{v_0v_1\ldots v_n}  x_{v_0v_1\ldots v_n}v_0v_1\ldots v_n,   \sum_{u_0u_1\ldots u_n}  y_{u_0u_1\ldots u_n} u_0u_1\ldots u_n\Big\rangle_g\\
&=&\sum_{v_0v_1\ldots v_n, u_0u_1\ldots u_n} x_{v_0v_1\ldots v_n}y_{u_0u_1\ldots u_n}\prod_{i=0}^n  g(v_i) g(u_i) \delta(v_i,u_i).
\end{eqnarray*}
 The dimension of $\Omega_n^{f,g}(G)$  and the dimension of  $\Omega_n(G)$ are equal.  The $(f,g)$-weighted Hodge-Laplace operator is a linear map
\begin{eqnarray*}
\Delta_n^{f,g}(G): \Omega_n^{f,g}(G)\longrightarrow \Omega_n^{f,g}(G)
\end{eqnarray*}
given  by
\begin{eqnarray*}
\Delta_n^{f,g}(G)(\omega)&=&\big(\partial_n^f\mid _{\Omega_n^{f,g}(G)}\big)^*\big(\partial_n^f\mid _{\Omega_n^{f,g}(G)}\big)(\omega) \nonumber\\
&&+\big(\partial_{n+1}^f\mid _{\Omega_{n+1}^{f,g}(G)}\big)\big(\partial_{n+1}^f\mid _{\Omega_{n+1}^{f,g}(G)}\big)^* (\omega)
\end{eqnarray*}
for any $\omega\in \Omega_n^{f,g}(G)$.   Let  $\{\lambda^{f,g}_i\}_{i=1}^{\dim\Omega_n(G)}$   be  the  eigenvalues  of $\Delta_n^{f,g}(G)$.    The zeta  function $\zeta_n^{f,g}$  of   $\Delta_n^{f,g}(G)$  is   defined by
 \begin{eqnarray*}
 \zeta_n^{f,g}(s)=\sum_{\lambda^{f,g}_i>0}\frac{1}{(\lambda^{f,g}_i)^s}.
 \end{eqnarray*}
 The  weighted analytic torsion  $T(G,f,g)$  of  $G$  is  given  by
 \begin{eqnarray}\label{eq-baqkl}
 \log  T(G,f,g)=\frac{1}{2}\sum_{n=0}^N(-1)^n n\frac{d}{ds}\Big|_{s=0}\Big(\zeta^{f,g}_n (s)\Big).
 \end{eqnarray}

For any real number $c\neq 0$ and any non-vanishing weight $f$ on $V$,  we  use $cf$  to denote the  non-vanishing    weight on $V$ with  the  value $cf(v)$ for each $v\in  V$.

Suppose there exists a positive integer $N$ satisfying $\Omega_n(G)=0$ for any integer $n\geq N+1$.  Let  $t(G,c)=|c|^{s(G)}$, where $s(G)=\sum_{n\geq 0} (-1)^n \dim\partial_n\Omega_n(G)$,   be  a function of $G$ and $c$.

\begin{theorem}\label{th-main}
Let $G=(V,E)$  be a digraph such that there exists a positive integer $N$ satisfying $\Omega_n(G)=0$ for any integer $n\geq N+1$.  Let $f$,  $g$ and $h$
 be three non-vanishing functions on $V$. Let $T(G,f,g)$  be the $(f,g)$-weighted analytic torsion of $G$.  Then
\begin{enumerate}[(i).]
\item
$T(G,fh,gh)=T(G,f,g)$;
\item
$T(G,f,cg)=t(G,c)T(G,f,g)$ and $T(G,cf,g)=t(G,c)^{-1}T(G,f,g)$ for  any non-zero real constant  $c$.
\end{enumerate}
\end{theorem}


  The ratio $f/g$  is a non-vanishing real function on $V$.  By Theorem~\ref{th-main}~(i),   the weighted analytic torsion $T(G,f,g)$ only depends on   $f/g$.  The next corollary  is a particular case of  Theorem~\ref{th-main}.

\begin{corollary}
Let $G$ be a digraph such that there exists a positive integer $N$ satisfying $\Omega_n(G)=0$ for any integer $n\geq N+1$.   Then  for any non-vanishing weight $f$ on $V$ and  any   real number $c\neq 0$,   we have
\begin{eqnarray*}
T(G)=T(G,f,f)=t(G,c)T(G,cf,f)
=t(G,c)^{-1}T(G,f,cf).
\end{eqnarray*}
\end{corollary}



\smallskip

\section{Some Preliminaries on Chain Complexes}\label{s2}

Firstly,  we  review some definitions on chain complexes  and the Hodge-Laplace  operators.  Secondly,  we give some preliminaries  on  the R-torsion and  the analytic torsion   of chain complexes.   This  is an analogue  of    \cite[Section~3]{lin1}.  Thirdly, we  prove some auxiliary  lemmas.

\smallskip

\subsection{Chain Complexes and The Hodge-Laplace Operators}

  Let $C_*=\{C_n,\partial_n\}_{n\geq 0}$ be a   chain complex where for each nonnegative integer  $n$,  $C_n$  is a  finite-dimensional   real vector space and $\partial_n: C_n\longrightarrow C_{n-1}$  is the boundary operator. Suppose  each    $C_n$ is equipped with an inner product $\langle~,~\rangle$.  Let $\partial_n^*: C_{n-1}\longrightarrow C_n$ be the adjoint operator of $\partial_n$ with respect to the inner products on $C_n$ and $C_{n-1}$.  The Hodge-Laplace operator $\Delta_n: C_n\longrightarrow C_n$  is defined by
\begin{eqnarray*}
\Delta_n u= \partial_n^*\partial_n u + \partial_{n+1}\partial_{n+1}^* u
\end{eqnarray*}
for any $u\in C_n$.
An element $u\in C_n$  is called harmonic if $\Delta_n u=0$.  Let $\mathcal{H}_n$  be the  set of all the  harmonic chains in $C_n$.  It can be proved that (cf. \cite[Section~3.1]{lin1})
\begin{eqnarray*}
\mathcal{H}_n={\rm Ker}\partial_n \cap {\rm Ker}\partial_{n+1}^*
\end{eqnarray*}
and the space $C_n$ is an orthogonal sum
\begin{eqnarray*}
C_n=\partial_{n+1} C_{n+1}\oplus \partial^*_{n}C_{n-1} \oplus\mathcal{H}_n
\end{eqnarray*}
of the  three subspaces $\partial_{n+1} C_{n+1}$,  $\partial^*_{n}C_{n-1}$  and  $\mathcal{H}_n$.
Moreover, there is a natural isomorphism
\begin{eqnarray}\label{eq-iso}
H_n(C_*)\cong \mathcal{H}_n
\end{eqnarray}
where  $H_n(C_*)$  is  the $n$-th homology of $C_*$.
It is direct to check that $\Delta_n$ is a self-adjoint semi-positive definite operator on $C_n$.  We denote its eigenvalues as $\{\lambda_i\mid 0\leq i\leq \dim C_n-1\}$.  The zeta function $\zeta_n(s)$  is defined by
\begin{eqnarray*}
\zeta_n(s)=\sum _{\lambda_i>0} \frac{1}{\lambda_i^s}.
\end{eqnarray*}
Suppose in addition  that there exists a positive integer $N$ such that $C_n=0$ for any  integer  $n\geq N+1$.  Then we have the next lemma.
\begin{lemma}\label{le-2.d}
For each  nonnegative integer $n$,  it holds
\begin{eqnarray*}
\sum_{n=0}^N(-1)^n n\Big(\dim C_n-\dim H_n\Big)=\sum_{n=0}^N(-1)^n\dim {\rm Im}\partial_n.
\end{eqnarray*}
\end{lemma}
\begin{proof}
For each   nonnegative integer  $n$, we note that $\dim H_n=\dim {\rm Ker}\partial_n-\dim{\rm Im}\partial_{n+1}$ and $\dim C_n=\dim {\rm Ker}\partial_n +\dim {\rm Im}\partial_{n}$.   By a direct calculation,  it follows that
\begin{eqnarray*}
\sum_{n=0}^N(-1)^n n\Big(\dim C_n-\dim H_n\Big)&=&\sum_{n=0}^N (-1)^n n\Big(\dim {\rm Im}\partial_n +\dim {\rm Im}\partial_{n+1}\Big)\\
&=&\sum_{n=0}^N \Big((-1)^n n +(-1)^{n-1}(n-1)\Big) \dim {\rm Im}\partial_n\\
&=&\sum_{n=0}^N(-1)^n\dim {\rm Im}\partial_n.
\end{eqnarray*}
We obtain the lemma.
\end{proof}

\smallskip

\subsection{The Analytic Torsions and The R-torsions}

The analytic torsion $T(C_*,\langle~,~\rangle)$ of the chain complex $C_*$  with  respect to the  inner product $\langle~,~\rangle$  is defined by
\begin{eqnarray*}
\log T(C_*,\langle~,~\rangle)=\frac{1}{2}\sum_{n=0}^N(-1)^n n \zeta'_n(0).
\end{eqnarray*}

On the other hand,  we can use the transition matrices between certain bases  (we  use bases as the plural of basis) to define the R-torsion $\tau$ (cf.  \cite[Section~3]{lin1}).  
For each  integer  $0\leq n\leq N$,  let ${\bf c}_n$  be a basis in $C_n$ and let ${\bf h}_n$ be a basis in  $H_n(C_*)$.     Let $B_n=\partial_{n+1} C_{n+1}$ and ${\bf  b}_n$ be any basis in $B_n$.  Let $Z_n={\rm Ker}(\partial_n)$.  Then $H_n(C_*)=Z_n/B_n$.    For each $w\in {\bf b}_{n-1}$, choose one element $v\in \partial_{n}^{-1} w$ such that $\partial_n v=w$.  Let $\tilde {\bf  b}_n$ be the collection of all these chains $v$.    Then $\tilde {\bf b}_n$  is a linearly  independent set in $C_n$. Similarly, for each element of ${\bf h}_n$ we choose its representative in $Z_n$ and denote the resulting independent set by $\tilde {\bf  h}_n$.  The union $({\bf  b}_n,\tilde {\bf  h}_n)$  is a basis in $Z_n$.  And the union $({\bf b}_n,\tilde {\bf  h}_n, \tilde {\bf  b}_n)$  is a basis in $C_n$.  By using the notation $[{\bf b}_n,\tilde{\bf h}_n,\tilde {\bf b}_n/{\bf  c}_n]$  to denote the absolute value  of the determinant of the transition matrix from  $({\bf b}_n,\tilde {\bf  h}_n, \tilde {\bf  b}_n)$  to ${\bf c}_n$,  the R-torsion $\tau(C_*,{\bf c},{\bf  h})$  of the chain complex $C_*$ with the preferred bases ${\bf c}$ and ${\bf  h}$ is a positive real number defined by (cf. \cite[Definition~3.6]{lin1})
\begin{eqnarray*}
\log \tau(C_*,{\bf c},{\bf h})=\sum_{n=0}^N(-1)^n \log [{\bf b}_n,\tilde{\bf h}_n,\tilde {\bf b}_n/{\bf c}_n].
\end{eqnarray*}
It is proved in \cite[Lemma~3.7]{lin1}  that the value of $\tau(C_*,{\bf c},{\bf h})$  does not depend on the choice of the bases ${\bf b}_n$,  the representatives in $\tilde {\bf b}_n$ and the representatives in $\tilde {\bf h}_n$; and if ${\bf c}'$ and ${\bf h}'$ are other collections of bases in $C_*$ and $H_*$ respectively,  then
\begin{eqnarray}\label{eq-lin-3.9}
\log \tau(C_*,{\bf c}',{\bf h}')=\log \tau(C_*,{\bf c},{\bf h})+\sum_{n=0}^N (-1)^n \big(\log [{\bf c}_n/{\bf c}'_n] +\log [{\bf h}'_n/{\bf h}_n]).
\end{eqnarray}

Let us fix an inner product $\langle~,~\rangle$ in $C_n$ for each integer  $0\leq n\leq N$. Then we have the induced inner product in the subspaces $B_n$, $Z_n$ and $\mathcal{H}_n$.  By (\ref{eq-iso}),  we represent each homology class by the corresponding harmonic chain. Thus we can transfer the inner product on the harmonic chains to $H_n(C_*)$.
We choose the bases ${\bf c}_n$ and ${\bf h}_n$ to be orthonormal and define the R-torsion of $(C_*,\langle~,~\rangle)$ by
\begin{eqnarray}\label{eq-lapq}
\tau(C_*,\langle~,~\rangle)=\tau(C_*,{\bf c},{\bf h}).
\end{eqnarray}
By (\ref{eq-lin-3.9}), the right-hand side does not depend on the choice of orthonormal bases ${\bf c}$ and ${\bf h}$.
It  is proved in \cite[Theorem~3.14]{lin1} that
\begin{eqnarray}\label{eq-id}
\tau(C_*,\langle~,~\rangle)=T(C_*,\langle~,~\rangle).
\end{eqnarray}

 Consider two different inner products $\langle~,~\rangle_1$  and $\langle~,~\rangle_2$ on $C_*$.  Assume that there are positive real numbers $r_n$, where $0\leq n\leq N$  are    integers,  such that for all chains $u,w\in C_n$,
\begin{eqnarray*}
\langle u,w\rangle_2 =r_n\langle u,w\rangle_1.
\end{eqnarray*}
Then by \cite[Corollary~3.8]{lin1} and (\ref{eq-id}),
\begin{eqnarray*}
T(C_*,\langle~,~\rangle_2)=T(C_*,\langle~,~\rangle_1)\prod_{n=0}^N r_n^{\frac{(-1)^n}{2}(\dim C_n-\dim H_n(C_*))}.
\end{eqnarray*}

\smallskip

\subsection{Auxiliary Lemmas}

The following lemmas   play  a key role in this paper.  

\begin{lemma}\label{le-bbab}
Let $C_*$  and $C'_*$  be two chain complexes.  For each   nonnegative integer $n$,  let $\langle~,~\rangle_n$  be an inner product on $C_n$ and let  $\langle~,~\rangle'_n$  be  an inner product  on $C'_n$.   If there is a chain isomorphism $\varphi: C_*\longrightarrow C'_*$ s.t.
\begin{eqnarray*}
\langle\varphi(a),\varphi(b)\rangle'_n=\langle a,b\rangle_n
\end{eqnarray*}
for any       nonnegative integer $n$ and any $a,b\in C_n$,
then for each   nonnegative integer  $n$,
\begin{eqnarray}\label{eq-key}
\Delta_n(C_*,\langle~,~\rangle_*)=\varphi^{-1}\circ \Delta_n(C'_*,\langle~,~\rangle'_*)\circ \varphi
\end{eqnarray}
where $\Delta_n(C_*,\langle~,~\rangle_*)$  is the Hodge-Laplace operator of the chain complex $C_*$ with respect to the inner products $\langle~,~\rangle_*$ and $\Delta_n(C'_*,\langle~,~\rangle'_*)$  is the Hodge-Laplace operator of the chain complex $C'_*$ with respect to the inner products $\langle~,~\rangle'_*$.
\end{lemma}
\begin{proof}
We write $\Delta_n(C_*,\langle~,~\rangle_*)$  as $\Delta_n$   and write $\Delta_n(C'_*,\langle~,~\rangle'_*)$ as $\Delta'_n$  for short.
Firstly,  we prove that for any  nonnegative integer $n$ and any  $a\in C_n$,
\begin{eqnarray}\label{eq-le2.1-1}
(\partial'_{n+1})^*\varphi(a)=\varphi(\partial^*_{n+1}(a)).
\end{eqnarray}
Note that since $\varphi$  is a chain  isomorphism,  we have that $\varphi$ is a linear isomorphism and for any   nonnegative integer  $n$,
\begin{eqnarray}\label{eq-le2.1-3}
\partial'_n\varphi=\varphi\partial_n.
\end{eqnarray}
Thus  for any $e\in C_{n+1}$,  we have
\begin{eqnarray*}
\langle (\partial'_{n+1})^*\varphi(a),\varphi(e)\rangle'&=& \langle \varphi(a),\partial'_{n+1}\varphi(e)\rangle'\\
&=&\langle \varphi(a),\varphi\partial_{n+1}(e)\rangle'\\
&=&\langle a, \partial_{n+1}(e)\rangle\\
&=&\langle \partial_{n+1}^* (a),e\rangle\\
&=&\langle \varphi\partial_{n+1}^* (a),\varphi(e)\rangle'.
\end{eqnarray*}
Letting $e$ run over $C_{n+1}$,   we note that $\varphi(e)$ runs over $C'_{n+1}$.  Thus (\ref{eq-le2.1-1}) follows  by the last equation.

 Secondly,  we prove that for any   nonnegative integer  $n$  and any $a\in C_n$,
\begin{eqnarray}\label{eq-le2.1-2}
 \varphi^{-1} \Delta'_n\varphi(a) =\Delta_n(a).
\end{eqnarray}
Let $b\in C_n$.  By a straight-forward calculation and with the help of   (\ref{eq-le2.1-1})  and  (\ref{eq-le2.1-3}),
\begin{eqnarray*}
\langle \varphi^{-1} \Delta'_n\varphi(a),b\rangle&=&\langle \Delta'_n\varphi(a),\varphi(b)\rangle'\\
&=&\langle (\partial'_n)^*\partial'_n(\varphi(a)),\varphi(b)\rangle' + \langle \partial'_{n+1}(\partial'_{n+1})^* (\varphi(a)),\varphi(b)\rangle' \\
&=&\langle (\partial'_n)^*\varphi\partial_n(a),\varphi(b)\rangle' + \langle \partial'_{n+1}\varphi(\partial_{n+1})^* (a),\varphi(b)\rangle' \\
&=&\langle \varphi(\partial_n)^*\partial_n(a),\varphi(b)\rangle' + \langle \varphi\partial_{n+1}(\partial_{n+1})^* (a),\varphi(b)\rangle' \\
&=&\langle (\partial_n)^*\partial_n(a),b\rangle + \langle \partial_{n+1}(\partial_{n+1})^* (a),b\rangle \\
&=&\langle \Delta_n(a),b\rangle.
\end{eqnarray*}
Letting $b$ run over $C_n$,  (\ref{eq-le2.1-2}) follows  from the last equation.   By (\ref{eq-le2.1-2}),  we have (\ref{eq-key}).  The lemma is proved.
\end{proof}

The  next lemma is a consequence of Lemma~\ref{le-bbab}.

\begin{lemma}\label{le-2.1}
Suppose all the assumptions in Lemma~\ref{le-bbab} are satisfied.  Then
by choosing  the  orthonormal bases properly,  the matrix representatives of the Hodge-Laplace operators are equal
\begin{eqnarray*}
[\Delta_n(C_*,\langle~,~\rangle_*)]=[ \Delta_n(C'_*,\langle~,~\rangle'_*)].
\end{eqnarray*}
Consequently,
\begin{eqnarray*}
T(C_*,\langle~,~\rangle_*)=T(C'_*,\langle~,~\rangle'_*).
\end{eqnarray*}
\end{lemma}

\begin{proof}
 We observe that both  $\Delta_n$ and $\Delta'_n$ are semi-positive defniite and self-adjoint.   We enumerate 
 all the eigenvalues  $0\leq \lambda_0<  \lambda_1< \cdots  <  \lambda _{k}$  of $\Delta_n$   as  well as     all  the  eigenvalues  $0\leq \lambda'_0<  \lambda'_1< \cdots  \leq  \lambda' _{k'}$  of $\Delta'_n$   in  an  increasing order, 
  where each  eigenvalue  $\lambda=\lambda_i$,  $i=1,2,\ldots,k$,   of  $\Delta_n$  as  well as   each eigenvalue  $\lambda'=\lambda'_j$,    $j=1,2,\ldots,k'$,   of $\Delta'_n$  is  counted with  multiplicity.   Let $E(\lambda)$ and $E'(\lambda')$ be the corresponding  eigenspaces of $\lambda$  and $\lambda'$ respectively.  Then
\begin{eqnarray*}
C_n= \oplus _{i=0}^k E(\lambda_i),~~~ C'_n=\oplus_{i=0}^{k'}E'(\lambda'_i).
\end{eqnarray*}
Let $e\in E'(\lambda'_i)$.  Then by  (\ref{eq-key}),
\begin{eqnarray*}
\Delta_n(\varphi^{-1}(e))=\varphi^{-1}\Delta'_n(e)=\varphi^{-1}(\lambda'_i e)=\lambda'_i\varphi^{-1}(e).
\end{eqnarray*}
Thus $\lambda_i'$  is also an eigenvalue of $\Delta_n$  and  $\varphi^{-1}(e)\in  E(\lambda'_i)$.  Conversely,  it can be proved in a similar way that for any $e\in E(\lambda_i)$, $\lambda_i$  is also an eigenvalue of $\Delta'_n$  and $\varphi(e)\in E'(\lambda_i)$.  Hence  the sets of eigenvalues $\{\lambda_i  \mid i=0,1,\ldots\}$ and  $\{\lambda'_i  \mid i=0,1,\ldots\}$  are equal, and $E(\lambda)$ is linearly isomorphic to $E'(\lambda)$ for each such eigenvalue $\lambda$.  Therefore,  as multi-sets, $\{\lambda_i \mid i=0,1,\ldots\}$ and  $\{\lambda'_i    \mid i=0,1,\ldots\}$ are equal.  The analytic torsions of $\Delta_n$ and $\Delta'_n$ must be  equal as well.
 The lemma follows.
\end{proof}

\smallskip

\section{Paths  on Discrete Sets}

 Let $V$  be a discrete set.
  Let $g: V\longrightarrow \mathbb{R}^\times$ be a non-vanishing real-valued weight function on $V$ assigning a non-zero number to each vertex  $v\in  V$.   Let $n$ be a nonnegative integer.  Let $\Lambda _n(V)$  be the vector space of  the  $n$-paths on $V$.  Then $\Lambda _n(V)$  has a basis 
  \begin{eqnarray*}
\{v_0 v_1\ldots v_n \mid  v_0,v_1,\ldots,v_n\in V\}
\end{eqnarray*}
consisting of all the  elementary $n$-paths.   
 The weight $g$ on $V$ induces a weight $g: \Lambda _n(V)\longrightarrow \mathbb{R}$  by
\begin{eqnarray*}
g\Big ( \sum  x_{v_0v_1\ldots v_n} v_0v_1\ldots v_n \Big)=\sum x_{v_0v_1\ldots v_n} g(v_0)g(v_1)\ldots g(v_n).
\end{eqnarray*}
We note that
\begin{eqnarray*}
g(v_0)g(v_1)\ldots g(v_n)\neq 0
\end{eqnarray*}
for any  elementary $n$-path $v_0v_1\ldots v_n$  on  $V$.
For any two vertices $v,u\in V$,  we use the notation  $\delta(v,u)=0$  if $v\neq u$ and $\delta(v,u)=1$  if $v=u$.  The weight $g$ induces an  inner product
\begin{eqnarray*}
\langle~,~\rangle_g:  \Lambda _n(V)\times \Lambda _n(V)\longrightarrow \mathbb{R}
\end{eqnarray*}
by
\begin{eqnarray*}
\langle v_0v_1\ldots v_n,    u_0u_1\ldots u_n\rangle_g=\prod_{i=0}^n  g(v_i) g(u_i) \delta(v_i,u_i), 
\end{eqnarray*}
where   the  inner product $\langle~,~\rangle_g$  extends bilinearly over  the real numbers.  Particularly,  we let $\langle~,~\rangle$  be the usual (un-weighted)  inner product on $\Lambda_n(V)$ given by the constant weight $g=1$.

Let $f: V\longrightarrow\mathbb{R}^\times$  be another non-vanishing real-valued weight function on $V$ (here the choices of $f$ and $g$  do not depend on each other). The $f$-weighted boundary map
\begin{eqnarray*}
\partial_n^f: \Lambda _n(V)\longrightarrow \Lambda _{n-1}(V)
\end{eqnarray*}
 is given by
\begin{eqnarray*}
\partial_n^f(v_0v_1\ldots v_n)=\sum_{i=0}^n (-1)^i f(v_i) v_0\ldots \widehat{v_i}
\ldots v_n,
\end{eqnarray*}
where  the map $\partial_n^f$  extends linearly over the real numbers.
The adjoint operator $(\partial_n^f)^*_g: \Lambda _{n-1}(V)\longrightarrow \Lambda _n(V)$ of $\partial_n^f$ with respect to $\langle~,~\rangle_g$  is given by
\begin{eqnarray*}
\langle\partial_n^f (a),b\rangle_g=\langle a,(\partial_n^f)^*_g(b)\rangle_g
\end{eqnarray*}
for any $a\in \Lambda _n(V)$ and any $b\in \Lambda _{n-1}(V)$.  The $n$-th  $(f,g)$-weighted Hodge-Laplace operator on $V$  is a linear map
\begin{eqnarray*}
\Delta_n^{f,g}:  \Lambda _n(V)\longrightarrow \Lambda _n(V)
\end{eqnarray*}
given by
\begin{eqnarray*}
\Delta_n^{f,g}=(\partial_n^f)^*_g(\partial_n^f) + (\partial_{n+1}^f)(\partial_{n+1}^f)^*_g.
\end{eqnarray*}
We use $\Lambda _*(V,f,g)$  to denote the chain complex $\{\Lambda _n(V),\partial_n^f\}_{n\geq 0}$ with the inner product $\langle~,~\rangle_g$.  Particularly when $f=g$,  we  have the next lemma.

\begin{lemma}\label{le-3.a}
For any three non-vanishing weights $f$, $g$ and $h$ on $V$,  there is a canonical chain isomorphism
\begin{eqnarray}\label{eq-mod}
\varphi: \Lambda _*(V,f,g)\longrightarrow  \Lambda _*(V,fh,gh)
\end{eqnarray}
given by
\begin{eqnarray}\label{eq-rep}
\varphi\Big(  \sum  x_{v_0v_1\ldots v_n} v_0v_1\ldots v_n \Big)=\sum  \frac{x_{v_0v_1\ldots v_n}}{h(v_0)h(v_1)\cdots h(v_n)} v_0v_1\ldots v_n
\end{eqnarray}
such that
\begin{eqnarray*}
\langle\varphi(a),\varphi(b)\rangle_{gh}=\langle a,b\rangle_g
\end{eqnarray*}
for any $a,b\in \Lambda _n(V)$.
\end{lemma}

\begin{proof}
Firstly,  let $\sum  x_{v_0v_1\ldots v_n} v_0v_1\ldots v_n$ be in $\Lambda _n(V)$.  We have
\begin{eqnarray*}
&&\partial_n^{fh}\circ \varphi \Big(\sum  x_{v_0v_1\ldots v_n} v_0v_1\ldots v_n\Big)\\
&=&\partial_n^{fh} \Big(\sum  \frac{x_{v_0v_1\ldots v_n}}{h(v_0)h(v_1)\cdots h(v_n)} v_0v_1\ldots v_n\Big)\\
&=&\sum  \frac{x_{v_0v_1\ldots v_n}}{h(v_0)h(v_1)\cdots h(v_n)} \sum_{i=1}^n (-1)^i  f(v_i)h(v_i) v_0\ldots \widehat{v_i}\cdots v_n\\
&=&\sum x_{v_0v_1\ldots v_n}\sum _{i=0}^n (-1)^i f(v_i) \frac{v_0\ldots \widehat{v_i}\cdots v_n}{h(v_0)\ldots \widehat{h(v_i)}\cdots h(v_n)}\\
&=&\varphi\circ \partial_n^f\Big(\sum  x_{v_0v_1\ldots v_n} v_0v_1\ldots v_n\Big).
\end{eqnarray*}
Thus $\varphi$ is a chain map.  Secondly, it is direct that $\varphi$ is a linear isomorphism.  Thirdly,  let $\sum  x_{v_0v_1\ldots v_n} v_0v_1\ldots v_n$ and  $\sum  y_{u_0u_1\ldots u_n} u_0u_1\ldots u_n$   be in $\Lambda _n(V)$.   Then
\begin{eqnarray*}
&&\Big\langle\varphi\Big(\sum  x_{v_0v_1\ldots v_n} v_0v_1\ldots v_n\Big),\varphi\Big(\sum  y_{u_0u_1\ldots u_n} u_0u_1\ldots u_n\Big)\Big\rangle_{gh}\\
&=&\sum\sum  \frac{x_{v_0v_1\ldots v_n}}{h(v_0)h(v_1)\cdots h(v_n)}\cdot  \frac{y_{u_0u_1\ldots u_n}}{h(u_0) h(u_1)\cdots h(u_n)}\Big\langle  v_0v_1\ldots v_n,u_0u_1\ldots u_n\Big\rangle_{gh}\\
&=&\sum\sum  \frac{x_{v_0v_1\ldots v_n}}{h(v_0)h(v_1)\cdots h(v_n)}\cdot  \frac{y_{u_0u_1\ldots u_n}}{h(u_0) h(u_1)\cdots h(u_n)}\prod_{i=0}^n g(v_i) h(v_i) g(u_i)  h(u_i) \delta(v_i,u_i)\\
&=&\sum\sum  x_{v_0v_1\ldots v_n}y_{u_0u_1\ldots u_n} \prod_{i=0}^n  g(v_i) g(u_i) \delta(v_i,u_i)\\
&=&\Big\langle \sum  x_{v_0v_1\ldots v_n} v_0v_1\ldots v_n,  \sum  y_{u_0u_1\ldots u_n} u_0u_1\ldots u_n\Big\rangle_g.
\end{eqnarray*}
Thus $\varphi$ is an isometry with respect to the inner products $\langle~,~\rangle_g$  and $\langle~,~\rangle_{gh}$.
\end{proof}

As a particular case  we take $f=g=1$ to be the trivial weight of  constant $1$   in Lemma~\ref{le-3.a}.   We have the next corollary.

\begin{corollary}
For any non-vanishing weight $h$ on $V$,  there is a canonical chain isomorphism
\begin{eqnarray*}
\varphi: \Lambda _*(V)\longrightarrow  \Lambda _*(V,h,h)
\end{eqnarray*}
given by
\begin{eqnarray*}
\varphi\Big(  \sum  x_{v_0v_1\ldots v_n} v_0v_1\ldots v_n \Big)=\sum  \frac{x_{v_0v_1\ldots v_n}}{h(v_0)h(v_1)\cdots h(v_n)} v_0v_1\ldots v_n
\end{eqnarray*}
such that
\begin{eqnarray*}
\langle\varphi(a),\varphi(b)\rangle_h=\langle a,b\rangle
\end{eqnarray*}
for any $a,b\in \Lambda _n(V)$. \qed
\end{corollary}

Note  that $\varphi$ in (\ref{eq-mod}) sends a  regular $n$-path to a regular $n$-path.  The next lemma follows from Lemma~\ref{le-3.a} and  transits the assertion  in Lemma~\ref{le-3.a} to  the quotient spaces $\mathcal{R}_*(V)$.

\begin{lemma}\label{le-3.aaa}
For any three non-vanishing weights $f$, $g$ and $h$ on $V$,  there is a canonical chain isomorphism
\begin{eqnarray}\label{eq-r.1}
\varphi: \mathcal{R} _*(V,f,g)\longrightarrow  \mathcal{R} _*(V,fh,gh)
\end{eqnarray}
given by (\ref{eq-rep})
such that
\begin{eqnarray*}
\langle \varphi(a),\varphi(b)\rangle_{gh}=\langle a,b\rangle_g
\end{eqnarray*}
for any $a,b\in \mathcal{R} _n(V,f,g)$. \qed
\end{lemma}

\smallskip

 \section{Proof of Theorem~\ref{th-main}}\label{s3}

\subsection{Proof of Theorem~\ref{th-main}~(i)}

 The next lemma follows   by  restricting  the chain isomorphism $\varphi$  in  Lemma~\ref{le-3.aaa} to  the sub-chain complex $\Omega_*^{f,g}(G)$  of $\mathcal{R}_*(V,f,g)$.

 \begin{lemma}\label{le-3.aaaaa}
For any non-vanishing weights $f$,  $g$  and  $h$ on $V$,  there is a canonical chain isomorphism
\begin{eqnarray*}
\varphi: \Omega _*^{f,g}(G)\longrightarrow  \Omega^{fh,gh} _*(G)
\end{eqnarray*}
given by (\ref{eq-rep})
such that
\begin{eqnarray*}
\langle \varphi(a),\varphi(b)\rangle_{gh}=\langle a,b\rangle_g
\end{eqnarray*}
for any $a,b\in \Omega_*(G)$.  \qed
\end{lemma}

\begin{proof}
With the help of the proof of Lemma~\ref{le-3.a} and Lemma~\ref{le-3.aaa}, we have
 \begin{eqnarray*}
\varphi(\mathcal{A}_n(G))
=\mathcal{A}_n(G).
\end{eqnarray*}
Moreover,  for any non-vanishing weights $f$,  $g$ and $h$ on $V$ and any
\begin{eqnarray*}
\sum x_{v_0v_1\ldots v_n} v_0v_1\ldots v_n\in \mathcal{A}_n(G),
\end{eqnarray*}
we have
\begin{eqnarray*}
&~~~~&\sum x_{v_0v_1\ldots v_n} v_0v_1\ldots v_n  \in \Omega^{f,g}_n(G)\\
&\Longleftrightarrow&\partial^f_n\Big(\sum x_{v_0v_1\ldots v_n} v_0v_1\ldots v_n\Big)\in \mathcal{A}_{n-1}(G)\\
&\Longleftrightarrow&
\varphi\partial_n^f\Big(\sum  x_{v_0v_1\ldots v_n} v_0v_1\ldots v_n\Big)\in \mathcal{A}_{n-1}(G)\\
&\Longleftrightarrow& \partial_n^{fh}\varphi \Big(\sum x_{v_0v_1\ldots v_n} v_0v_1\ldots v_n\Big)\in \mathcal{A}_{n-1}(G)\\
&\Longleftrightarrow&\varphi\Big(\sum x_{v_0v_1\ldots v_n} v_0v_1\ldots v_n\Big)  \in \Omega^{fh,gh}_n(G).
 \end{eqnarray*}
Therefore,  we observe that the map (\ref{eq-r.1}) sends $\Omega_n^{f,g}(G)$ to $\Omega_n^{fh,gh}(G)$ isomorphically  for each $n$ be a nonnegative integer.
The lemma follows.
\end{proof}

By Lemma~\ref{le-2.1} and Lemma~\ref{le-3.aaaaa},  we have

\begin{proposition}\label{pr-3.8}
For any digraph $G$ and any non-vanishing weights $f$,  $g$ and $h$ on $V$,  we have
\begin{eqnarray*}
\Delta_n^{f,g}(G)=\varphi^{-1}\circ \Delta^{fh,gh}_n(G)\circ \varphi
\end{eqnarray*}
where $\varphi$  is the isomorphism given in Lemma~\ref{le-3.a}.  Consequently,  by choosing  the  orthonormal bases properly,  the matrix representatives satisfy
\begin{eqnarray}\label{eq-add1}
[\Delta_n^{f,g}(G)]=[\Delta^{fh,gh}_n(G)]
\end{eqnarray}
for each $n$ be a nonnegative integer. In addition,  if   there exists a positive integer $N$ such that $\Omega_n(G)=0$ for any $n\geq N+1$,  then it follows from (\ref{eq-add1}) that
\begin{eqnarray}\label{eq-3.aaacc}
T(G,f,g)=T(G,fh,gh).
\end{eqnarray}
\qed
\end{proposition}

\begin{remark}
For a digraph $G$,  if there are infinitely many $n$ such that $\Omega_n(G)\neq 0$,  then the analytic torsion $T$ is not well-defined.   In this case, we cannot get the last assertion   (\ref{eq-3.aaacc})  in Proposition~\ref{pr-3.8}.
\end{remark}

\begin{remark}
By   Proposition~\ref{pr-3.8},  the $(f,g)$-weighted Hodge-Laplace operator $\Delta_n^{f,g}(G)$ only depends on the ratio function $f/g$.  Thus in order to study the $(f,g)$-weighted Hodge-Laplace operators $\Delta_n^{f,g}(G)$ for all the  non-vanishing weights $f$ and $g$,  it is equivalent to study $\Delta^{f,1}_n(G)$ for all the  non-vanishing weights $f$  or study $\Delta_n^{1,g}(G)$ for all the  non-vanishing weights $g$.
\end{remark}

Following from (\ref{eq-add1})  in Proposition~\ref{pr-3.8},  we have

\begin{corollary}\label{co-3.n}
For any digraph $G$ and any non-vanishing weights $f$ and $g$ on $V$,    $\varphi$ induces a canonical linear isomorphism
\begin{eqnarray}\label{eq-3.b}
\varphi: H_n(\{\Omega^{f,g}_k(G), \partial_k^f\}_{k\geq 0})\longrightarrow H_n(G),~~~n=0,1,2,\ldots.
\end{eqnarray}
 Here $H_n(G)$ is the usual (un-weighted) path homology of $G$ with coefficients in real numbers.  Moreover, if  $f=g$,  then  by  representing  the  homology classes  by the  corresponding  harmonic chains,  we have that   (\ref{eq-3.b}) is an isometry  \footnote[2]{The weighted homology  (with coefficients in a field)  of  a weighted simplicial complex  (cf. \cite{dawson,  cy1,cy, cy2})   with non-vanishing weights  is  proved to be linearly isomorphic  to the usual simplicial homology  (cf. \cite[Theorem~5.1]{cy}). }.
  \qed
\end{corollary}

Theorem~\ref{th-main}~(i)  follows from  (\ref{eq-3.aaacc})  in Proposition~\ref{pr-3.8}  immediately.



\smallskip

\subsection{Proof of Theorem~\ref{th-main}~(ii)}

We fix $f$ and multiply $g$ by a non-zero scalar $c$.  By applying \cite[Corollary~3.8]{lin1} to the R-torsion,  we  have

\begin{proposition}\label{pr-3.v}
Let $G$ be a digraph.  Suppose   there exists a positive integer $N$ such that $\Omega_n(G)=0$ for any integer  $n\geq N+1$.   Let $f$ and $g$ be   non-vanishing weights on $V$.  For any   real number $c\neq 0$,  we  use $cg$  to denote the     weight on $V$ with the value $cg(v)$ at any $v\in V$. Then
\begin{eqnarray}\label{eq-3.9}
T(G,f,cg)=|c|^{s(G)}T(G,f,g)
\end{eqnarray}
where
\begin{eqnarray*}
s(G)=\sum_{n=0}^N(-1)^n \dim{\rm Im} \big(\partial_n\mid _{\Omega_n(G)}\big).
\end{eqnarray*}
\end{proposition}

\begin{proof}
For any $n$ be a nonnegative integer and any chains $\omega,\omega'\in \Omega_n^{f,cg}(G)$, we have
\begin{eqnarray*}
\langle\omega,\omega'\rangle_{cg}=c^{2(n+1)}\langle
\omega,\omega'\rangle_g.
\end{eqnarray*}
By  Proposition~\ref{pr-3.8} and  Corollary~\ref{co-3.n}, for each $n$ be a nonnegative integer we have
\begin{eqnarray*}
\dim \Omega_n^{f,cg}(G)&=&\dim \Omega_n^{f,g}(G)\\
&=&\dim \Omega_n (G),\\
\dim H_n(\{ \Omega^{f,cg}_k(G), \partial_k^f \}_{k\geq 0})&=&\dim H_n(\{\Omega^{f,g}_k(G), \partial_k^f\}_{k\geq 0})\\
&=&\dim H_n(G).
\end{eqnarray*}
 Thus  by  \cite[Corollary~3.8]{lin1},  we have (\ref{eq-3.9}) where
\begin{eqnarray}
t(G,c)&=&\prod_{n= 0}^N |c|^{(-1)^{n} (n+1) \big(\dim \Omega_n(G)-\dim H_n(G)\big)}\nonumber\\
&=&|c|^{\sum_{n=0}^N  (-1)^n (n+1) \big(\dim \Omega_n(G)-\dim H_n(G)\big)}\nonumber\\
&=&|c|^{\sum_{n=0}^N  (-1)^n n \big(\dim \Omega_n(G)-\dim H_n(G)\big)}. \label{eq-revise}
\end{eqnarray}
The last equality of (\ref{eq-revise}) follows from that
\begin{eqnarray*}
\chi(G)=\sum_{n=0}^N  (-1)^n   \dim \Omega_n(G)=\sum_{n=0}^N  (-1)^n \dim H_n(G).
\end{eqnarray*}
By Lemma~\ref{le-2.d}, the right-hand side of  the last equality in (\ref{eq-revise})  equals to $|c|^{s(G)}$.
\end{proof}

We  fix $g$  and multiply $f$  by a non-zero scalar $c$.  By a straight-forward calculation of the analytic torsion,  we  have

\begin{proposition}\label{pr-add}
Let $G$ be a digraph.  Suppose there exists a positive integer $N$ such that $\Omega_n(G)=0$ for any  integer  $n\geq N+1$.  Let $f$ and $g$ be   non-vanishing weights on $V$.  For any   real number $c\neq 0$,  we  use $cf$  to denote the     weight on $V$ with the value $cf(v)$ at any $v\in V$.  Then
\begin{eqnarray}\label{eq-3.29}
T(G,cf,g)=t(G,c)^{-1}T(G,f,g).
\end{eqnarray}
\end{proposition}

\begin{proof}
Let $n$ be a nonnegative integer.
It follows directly that
\begin{eqnarray}\label{eq-f-1}
\partial_n^{cf}=c\partial_n^f.
\end{eqnarray}
By a similar argument with  the proof of Lemma~\ref{le-2.1},     we have
\begin{eqnarray*}
\langle (\partial_{n+1}^{cf})^* u, w\rangle_g =c\langle (\partial_{n+1}^f)^* u,w\rangle_g
\end{eqnarray*}
for any chains $u,w\in\Omega_n(G)$.   This  implies that
\begin{eqnarray}\label{eq-f-2}
(\partial_{n+1}^{cf})^* =c (\partial_{n+1}^f)^*.
\end{eqnarray}
By (\ref{eq-f-1})  and (\ref{eq-f-2}),  we have
\begin{eqnarray*}
\Delta_n^{cf,g}=c^2\Delta_n^{f,g}.
\end{eqnarray*}
Consequently,  if we denote the non-zero eigenvalues of $\Delta_n^{f,g}$ as  $0< \lambda_1< \lambda_2<\ldots$,  then the non-zero  eigenvalues of $\Delta_n^{cf,g}$  are  $0< c^2 \lambda_1< c^2\lambda_2<\ldots$  where $c^2\lambda_i$  has the same multiplicity with $\lambda_i$.
Therefore,  with the help of \cite[(3.21)]{lin1},  we  have
\begin{eqnarray*}
\log T(G,cf, g)&=&
\frac{1}{2}\sum_{n=0}^N(-1)^n n \frac{d}{ds}\Big| _{s=0}\Big(\sum_{\lambda_i>0}\frac{1}{(c^2\lambda_i)^s}\Big)\\
&=&\frac{1}{2}\sum_{n=0}^N(-1)^n n
\big(-\sum_{\lambda_i>0}
\log(c^2\lambda_i)\big)\\
&=&\frac{1}{2}\sum_{n=0}^N(-1)^n n
\big(-\sum_{\lambda_i>0}\log (\lambda_i)
\big)\\
&&-\frac{1}{2}\sum_{n=0}^N(-1)^n n (2\log |c|)\big(\dim\Omega_n(G)-\dim H_n(G)\big)\\
&=&\frac{1}{2}\sum_{n=0}^N(-1)^n n \frac{d}{ds}\Big| _{s=0}\Big(\sum_{\lambda_i>0}\frac{1}{\lambda_i^s}\Big)
-(\log |c|)s(G)\\
&=& \log T(G,f,g)-(\log |c|)s(G).
\end{eqnarray*}
Taking the exponential map on both sides of the  equations,  we have
(\ref{eq-3.29})
\footnote[3]{ For the aim  of corroboration,  we proved Proposition~\ref{pr-3.v}  and Proposition~\ref{pr-add} independently and by different approaches.  Nevertheless,
\begin{itemize}
\item
combining Proposition~\ref{pr-3.8} and Proposition~\ref{pr-3.v} together,  we   can  obtain  Proposition~\ref{pr-add};
\item
 combining  Proposition~\ref{pr-3.8} and Proposition~\ref{pr-add}  together,   we  can obtain Proposition~\ref{pr-3.v}.
 \end{itemize}
  }.
\end{proof}



Summarizing  Proposition~\ref{pr-3.v}  and Proposition~\ref{pr-add},  we obtain Theorem~\ref{th-main}~(ii).

\smallskip

\section{Examples}\label{s4}

\begin{example}
Consider the line digraph $G=(V,E)$ given in \cite[Figure~1, Example~3.9]{lin1}  where $V=\{0,1,2,3,4\}$  and $E$ consists of $4$ directed edges where the $i$-th directed edge has the form either $(i,i+1)$ or $(i+1,i)$,  for $i=0,1,2,3$.  Let $f$ and $g$ be two non-vanishing real-valued weight functions on $V$.  As in \cite[(3.17)]{lin1},  we denote $\bar e_{i(i+1)}$ for the directed edge (which is called the $1$-path in our setting) $(i,i+1)$ or $(i+1,i)$  in $E$.  We  denote $e_i$ for the $0$-path  consisting of a single vertex $i$.  Let $\sigma_i=1$  if $(i,i+1)\in E$ and $-1$ if $(i+1,i)\in E$.  Then
\begin{eqnarray*}
\partial_1^f \bar e_{i(i+1)}=\sigma_i(f(i)e_{i+1}-f(i+1) e_i)
\end{eqnarray*}
and
\begin{eqnarray*}
&\langle e_i, e_j\rangle_g=g(i)g(j)\delta(i,j),\\
&\langle \bar e_{i(i+1)} , \bar e_{j(j+1)}\rangle _g =g(i)g(i+1)g(j)g(j+1)\delta(i,j).
\end{eqnarray*}
Note that the last equality holds because each edge is assigned with exactly one direction.
Choose the following $\langle~,~\rangle_g$-orthonormal basis
\begin{eqnarray*}
\omega_0=\Big\{\frac{e_i}{g(i)}\mid i=0,1,2,3\Big\}
\end{eqnarray*}
 in $\Omega_0^{f,g}(G)$ and   the following $\langle~,~\rangle_g$-orthonormal basis
\begin{eqnarray*}
\omega_1=\Big\{\frac{\bar  e_{i(i+1)}}{g(i)g(i+1)} \mid  i=0,1,2 \Big\}
\end{eqnarray*}
in  $\Omega_1^{f,g}(G)$.  We note  $\Omega_n^{f,g}(G)=0$  for $n\geq 2$.  In $\partial_1^f (\Omega^{f,g}_1(G))$ choose the basis
\begin{eqnarray*}
b_0=\Big\{\sigma_i\Big(\frac{f(i)}{g(i)}\frac{e_{i+1}}{g(i+1)}-\frac{f(i+1)}{g(i+1)} \frac{e_i}{g(i)}\Big)\mid  i=0,1,2\Big\}
\end{eqnarray*}
and set
\begin{eqnarray*}
\tilde b_1=\Big\{\frac{\bar  e_{i(i+1)}}{g(i)g(i+1)} \mid  i=0,1,2 \Big\}.
\end{eqnarray*}
It is clear that ${\rm Ker}\partial_0^f=\Omega_0^{f,g}(G)$.  Thus the $\langle~,~\rangle_g$-orthogonal complement of $\partial_1^f\Omega_1^{f,g}(G)$  in ${\rm Ker}\partial_0^f$  is
\begin{eqnarray*}
\mathcal{H}_0={\rm Span}\Big\{\frac{f(0)}{g(0)^2}e_0+\frac{f(1)}{g(1)^2}e_1+\frac{f(2)}{g(2)^2}e_2+\frac{f(3)}{g(3)^2}e_3\Big\}
\end{eqnarray*}
so that
\begin{eqnarray*}
h_0=\Big\{\frac{\frac{f(0)}{g(0)}\frac{e_0}{g(0)}+\frac{f(1)}{g(1)}\frac{e_1}{g(1)}+\frac{f(2)}{g(2)}\frac{e_2}{g(2)}+\frac{f(3)}{g(3)}\frac{e_3}{g(3)}}{\sqrt{\frac{f(0)^2}{g(0)^2}+\frac{f(1)^2}{g(1)^2}+\frac{f(2)^2}{g(2)^2}+\frac{f(3)^2}{g(3)^2}}}
\Big\}.
\end{eqnarray*}
 Similar with \cite[(3.19)]{lin1},  by taking the absolute value of the determinant of the transition matrix,  we see that
\begin{eqnarray*}
[b_0,h_0,\tilde b_0/\omega_0]&=&\frac{1}{\sqrt{\sum_{i=0}^3\frac{f(i)^2}{g(i)^2}}} \Big| \sum_{k=0}^3 (-1)^{k+1} \frac{f(k)}{g(k)}\prod_{0\leq j\leq k-1}  \frac{-\sigma_j f(j+1)}{g(j+1)}\prod_{k\leq j\leq 2} \frac{\sigma_jf(j)}{g(j)}\Big|\\
&=&\frac{1}{\sqrt{\sum_{i=0}^3\frac{f(i)^2}{g(i)^2}}} \Big|  \sum_{k=0}^3  (-1)^{2k+1}\frac{f(k)}{g(k)}\prod_{0\leq j\leq k-1}  \frac{\sigma_j f(j+1)}{g(j+1)}\prod_{k\leq j\leq  2} \frac{\sigma_jf(j)}{g(j)}\Big|\\
&=&\frac{1}{\sqrt{\sum_{i=0}^3\frac{f(i)^2}{g(i)^2}}} \Big|  \sum_{k=0}^3  \frac{f(k)}{g(k)}\prod_{0\leq j\leq k-1}  \frac{\sigma_j f(j+1)}{g(j+1)}\prod_{k\leq j\leq 2} \frac{\sigma_jf(j)}{g(j)}\Big|.
\end{eqnarray*}
Moreover, since $b_1$ is the empty-set,  we have
\begin{eqnarray*}
h_1=\frac{\bar  e_{34}}{g(3)g(4)}.
\end{eqnarray*}
Consequently,  by taking the absolute value of the determinant of the transition matrix,  we have
\begin{eqnarray*}
[b_1,h_1,\tilde b_1/\omega_1]=1.
\end{eqnarray*}
It follows that the weighted analytic torsion is
\begin{eqnarray*}
T(G,f,g)&=&\prod_{n=0}^1[b_n,h_n,\tilde b_n/\omega_n]^{(-1)^n}\\
&=&\frac{1}{\sqrt{\sum_{i=0}^3\frac{f(i)^2}{g(i)^2}}} \Big|  \sum_{k=0}^3   \frac{f(k)}{g(k)}\prod_{0\leq j\leq k-1}  \frac{\sigma_j f(j+1)}{g(j+1)}\prod_{k\leq j\leq 2} \frac{\sigma_jf(j)}{g(j)}\Big|.
\end{eqnarray*}
In particular,  consider the following cases:
\begin{enumerate}[(i).]
\item
$f=g$.

Then $T(G,f,f)$ reduces to $T(G,1,1)$ in \cite[Example~3.9]{lin1},  which takes value $\sqrt{3}$  and does not depend on the choice of $f$.

\item
$g$ takes a constant value $c\neq 0$.

Then by a direct calculation,  we  have $t(f,c)=|c|^{-3}$  and $T(G,f,c)=|c|^{-3}T(G,f,1)$  for any non-vanishing weight $f$.  On the other hand,  we note that
\begin{eqnarray*}
s(G)=(-1)^1\dim  \partial_1(\Omega_1(G))=-3.
\end{eqnarray*}

\item
$f$ takes a constant value $c\neq 0$.

Then by a direct calculation,  we  have $T(G,c,g)=|c|^{3}T(G,1,g)$  for any non-vanishing weight $g$.
\end{enumerate}

\end{example}

\begin{example}
Consider the triangle $G=(V,E)$ where $V=\{0,1,2\}$ and $E=\{01,12,02\}$ (cf. \cite[Example~3.10]{lin1}).  Let $f$ and $g$ be two non-vanishing real  functions on $V$.  We  have
\begin{eqnarray*}
&\Omega^{f,g}_0(G)={\rm Span}\{e_0,e_1,e_2\}, ~~\Omega_1^{f,g}(G)={\rm Span}\{e_{01}, e_{12},e_{02}\},~~~\Omega^{f,g}_2(G)={\rm Span}\{e_{012}\},\\
&\partial^f_1\Omega_1^{f,g}(G)={\rm Span}\{f(0)e_1-f(1)e_0, f(1)e_2-f(2)e_1, f(0)e_2-f(2)e_0\},\\
&\partial^f_2\Omega^{f,g}_2(G)={\rm Span}\{f(0)e_{12} -f(1)e_{02} + f(2) e_{01}\}
\end{eqnarray*}
and $\Omega_n^{f,g}(G)=0$  for any integer  $n\geq 3$.  We choose the following $\langle~,~\rangle_g$-orthonormal basis
\begin{eqnarray*}
\omega_0=\Big\{ \frac{e_0}{g(0)}, \frac{e_1}{g(1)},\frac{e_2}{g(2)}    \Big\}
\end{eqnarray*}
in $\Omega_0^{f,g}(G)$,   the following $\langle~,~\rangle_g$-orthonormal basis
\begin{eqnarray*}
\omega_1=\Big\{ \frac{e_{01}}{g(0)g(1)}, \frac{e_{12}}{g(1)g(2)},\frac{e_{02}}{g(0)g(2)}    \Big\}
\end{eqnarray*}
in $\Omega_1^{f,g}(G)$, and   the following $\langle~,~\rangle_g$-orthonormal basis
\begin{eqnarray*}
\omega_2=\Big\{ \frac{e_{012}}{g(0)g(1)g(2)} \Big\}
\end{eqnarray*}
in $\Omega_2^{f,g}(G)$.
Choose also the basis
\begin{eqnarray*}
b_0=\Big\{\frac{f(0)e_1-f(1)e_0}{g(0)g(1)}, \frac{f(1)e_2-f(2)e_1}{g(1)g(2)}
  \Big\}
\end{eqnarray*}
in $\partial_1^f\Omega^{f,g}_1(G)$ and  the basis
\begin{eqnarray*}
b_1=\Big\{\frac{f(0)e_{12}-f(1)e_{02}+ f(2)e_{01} }{g(0)g(1)g(2)}
  \Big\}
\end{eqnarray*}
in $\partial_2^f\Omega^{f,g}_2(G)$.  Then their lifts are
\begin{eqnarray*}
\tilde b_1= \Big\{\frac{e_{01}}{g(0)g(1)},\frac{e_{12}}{g(1)g(2)} \Big\}
\end{eqnarray*}
and
\begin{eqnarray*}
\tilde b_2= \Big\{\frac{e_{012}}{g(0)g(1)g(2)}  \Big\}.
\end{eqnarray*}
Hence representing the homology classes by  the  corresponding  harmonic chains and taking the orthogonal complement of $\partial_1^{f,g}\Omega^{f,g}_1(G)$  in ${\rm  Ker}\partial^f_0$ with respect to $\langle~,~\rangle_g$,  we have
\begin{eqnarray*}
\mathcal{H}_0={\rm Span}\Big\{ \frac{f(0)}{g(0)^2}e_0 +\frac{f(1)}{g(1)^2}e_1       + \frac{f(2)}{g(2)^2}e_2 \Big\}
\end{eqnarray*}
so that
\begin{eqnarray*}
h_0=\Big\{  \frac{1}{\sqrt{\sum_{i=0}^2\frac{f(i)^2}{g(i)^2}}} \Big(\frac{f(0)}{g(0)^2}e_0 +\frac{f(1)}{g(1)^2}e_1       + \frac{f(2)}{g(2)^2}e_2 \Big) \Big\}.
\end{eqnarray*}
Moreover, note that  the union of $b_1$ and $\tilde b_1$ spans  $\Omega_1^{f,g}(G)$.  Thus   $h_1$  is the empty-set.  Similarly, both $b_2$ and $h_2$ are empty-sets.  By taking the absolute values of the determinants of the transition matrices,
it follows that
\begin{eqnarray*}
&& [b_0,h_0,\tilde b_0/\omega_0] = \Big|\frac{f(1)}{g(1)}\Big|\cdot  \sqrt{\dfrac{f(0)^2}{g(0)^2}+\dfrac{f(1)^2}{g(1)^2}+\dfrac{f(2)^2}{g(2)^2} }, \\
&& [b_1,h_1,\tilde b_1/\omega_1] = \dfrac{|f(1)|}{|g(1)|}, \\
&& [b_2,h_2,\tilde b_2/\omega_2] = 1.
\end{eqnarray*}
Hence we obtain
\begin{eqnarray*}
T(G,f,g)&=&\prod_{n=0}^2 [b_n,h_n,\tilde b_n/\omega_n]^{(-1)^n}\\
&=& \sqrt{\frac{f(0)^2}{g(0)^2}+\frac{f(1)^2}{g(1)^2}+\frac{f(2)^2}{g(2)^2}}.
\end{eqnarray*}
In particular, consider the following cases:
\begin{enumerate}[(i).]
\item
$f=g$.

  Then $T(G,f,f)=\sqrt{3}$  which is the same as \cite[Example~3.10]{lin1} and does not depend on the choice of $f$.

\item
$g$ takes a constant value $c\neq 0$.

 Then by a direct calculation,  we  have $t(f,c)=|c|^{-1}$  and
$T(G,f,c)=|c|^{-1}T(G,f,1)$  for any non-vanishing weight $f$.  On the other hand, we note that
\begin{eqnarray*}
s(G)=(-1)^1 \dim\partial_1\Omega_1(G)
+(-1)^2\dim\partial_2\Omega_2(G)=-2+1=-1.
\end{eqnarray*}

\item
$f$ takes a constant value $c\neq 0$.

Then by a direct calculation,  we  have $T(G,c,g)=|c|T(G,1,g)$  for any non-vanishing weight $g$.
\end{enumerate}
\end{example}

\smallskip

\section{Further Discussions}

 {\sc Discussion~1}.  Let $f$ and $g$  be any  non-vanishing functions on $V$.
It  follows from Lemma~\ref{le-bbab}   and  Proposition~\ref{pr-3.8}  that for each   nonnegative integer  $n$,  the matrix representation  of the $(f,g)$-weighted Hodge-Laplace operator   is  given by
\begin{eqnarray*}
&&[\Delta_n^{f,g}(G)]= [\Delta_n^{f/g,1}(G)]\\
&=& \Big[\big(\partial_n^{f/g}\mid _{\Omega_n^{f/g,1}(G)}\big)^*\Big]\Big[\big(\partial_n^{f/g}\mid _{\Omega_n^{f/g,1}(G)}\big)\Big]  +\Big[\big(\partial_{n+1}^{f/g}\mid _{\Omega_{n+1}^{f/g,1}(G)}\big)\Big]\Big[\big(\partial_{n+1}^{f/g}\mid _{\Omega_{n+1}^{f/g,1}(G)}\big)^*\Big].
\end{eqnarray*}
Here the adjoints of the boundary maps  are with respect to the canonical  inner product
\begin{eqnarray*}
&&\Big\langle \sum_{v_0v_1\ldots v_n}  x_{v_0v_1\ldots v_n}v_0v_1\ldots v_n,   \sum_{u_0u_1\ldots u_n}  y_{u_0u_1\ldots u_n} u_0u_1\ldots u_n\Big\rangle \\
&=&\sum_{v_0v_1\ldots v_n, u_0u_1\ldots u_n} x_{v_0v_1\ldots v_n}y_{u_0u_1\ldots u_n}\prod_{i=0}^n    \delta(v_i,u_i).
\end{eqnarray*}
It  is  prospective to study   the relations between    the spectra of $\Delta_n^{f,g}(G)$  and the  ratio  of weights  $f/g$.

\smallskip

 \noindent {\sc Discussion~2}.   In  \cite{lin3},   the fundamental group $\pi_1(G)$  of a digraph $G$ was defined.   It is prospective to generalize the definitions of the analytic torsions and the R-torsions of digraphs  by considering the orthogonal representations of $\pi_1(G)$ and the deck transformations  of $\pi_1(G)$ on the covering digraphs.  It is also prospective to study  the analytic torsions and the R-torsions  on weighted digraphs  with respect to the orthogonal representations of $\pi_1(G)$  after the results in  \cite{lin1}.

\bigskip

\section*{Acknowledgement}   The  authors would like to express their  deep
gratitude to the referee for the  careful reading of the manuscript.

{\small

\bigskip

Shiquan Ren  (first author)

Address:   School of Mathematics and Statistics,  Henan University,  Kaifeng  475004,  China.

E-mail:  renshiquan@henu.edu.cn

\bigskip

Chong Wang  (corresponding author)

Address:  School
of Mathematics and Statistics, Cangzhou Normal University, Cangzhou, Hebei, 061000, P. R. China.

E-mail:  wangchong\_618@163.com

}

\end{document}